\documentclass[12pt]{amsart}
\usepackage{graphicx, amsfonts, amssymb}
\usepackage{mathrsfs, amsmath, amsthm}
\usepackage{hyperref}

\def\ep{\varepsilon}

\newcommand{\D}{\mathbb{D}}

\newcommand{\T}{\mathbb{T}}
\newcommand{\N}{\mathbb{N}}

\renewcommand{\H}{\mathcal{H}}

\newcommand{\B}{\mathcal{B}}

\newcommand{\vp}{\vp}

\newcommand{\og}{\mathrm{O}}
\newcommand{\op}{\mathrm{o}}

\def\a{\alpha}       \def\b{\beta}        \def\g{\gamma}
           
     \def\om{\omega}      
       \def\t{\theta}       
                  \def\z{\zeta}
                  \def\vp{\varphi}

\newcommand{\hg}{\mathcal{H}_{g}}
 
\def\BMOA{\mathord{\rm BMOA}}
\def\VMOA{\mathord{\rm VMOA}}

\def\Dp{{\mathcal D^p_{p-1}}}

\def\Dpa{{\mathcal D^p_{\alpha}}}
\def\Dq{{\mathcal D^q_{q-1}}}
\newtheorem{lettertheorem}{Theorem}

\newtheorem{letterlemma}[lettertheorem]{Lemma}

\newtheorem{defin}{Definition}

\newtheorem{theorem}[defin]{Theorem}
\newtheorem{lemma}[defin]{Lemma}
\newtheorem{proposition}[defin]{Proposition}

\newtheorem{corollary}[defin]{Corollary}

\newtheorem{qu}{Question}

\newenvironment{Prf}{\noindent{\emph{Proof of}}}
{\hfill$\Box$ }

\numberwithin{equation}{section}

\begin{document}

\title[Operator theoretic differences ]
{Operator theoretic differences between Hardy and Dirichlet-type spaces }
\author{Jos\'e \'Angel Pel\'aez}
\address{Departamento de An´alisis Matem´atico, Universidad de M´alaga, Campus de
Teatinos, 29071 M´alaga, Spain} \email{japelaez@uma.es}

\author{F.~P\'erez-Gonz\'alez}
\address{Departamento de An\'{a}lisis Matem\'{a}tico,
Universidad de La Laguna, 38271 La Laguna, Tenerife, Spain}
\email{fernando.perez.gonzalez@ull.es}

\author{Jouni R\"atty\"a}
\address{University of Eastern Finland, P.O.Box 111, 80101 Joensuu, Finland}
\email{jouni.rattya@uef.fi}

\thanks{This research was supported in part by the Ram\'on y Cajal program
of MICINN (Spain), Ministerio de Edu\-ca\-ci\'on y Ciencia, Spain,
(MTM2011-25502), from La Junta de Andaluc{\'i}a, (FQM210) and
(P09-FQM-4468), MICINN- Spain ref.
MTM2011-26538.}

\date{\today}

\maketitle

\begin{abstract}
For $0<p<\infty $, the Dirichlet-type space $\Dp$ consists of those analytic functions
$f$ in the unit disc $\D$ such that $\int_\D\vert f'(z)\vert\sp
p(1-|z|)^{p-1}\,dA(z)<\infty$.
Motivated by operator theoretic differences between the Hardy space $H^p$ and $\Dp$, the integral operator
\begin{displaymath}
    T_g(f)(z)=\int_{0}^{z}f(\zeta)\,g'(\zeta)\,d\zeta,\quad
    z\in\D,
    \end{displaymath}
acting from one of these spaces to another is studied.
In particular, it is shown, on one hand, that $T_g:\Dp\to H^p$ is bounded if and only if $g\in\BMOA$ when $0<p\le 2$, and, on the other hand, that this equivalence is very far from being true if $p>2$.
Those symbols $g$ such that $T_g:\Dp\to H^q$ is bounded (or compact) when $p<q$ are also characterized. Moreover, the best known sufficient $L^\infty$-type condition for a positive Borel measure $\mu$ on $\D$ to be a $p$-Carleson measures for $\Dp$, $p>2$, is significantly relaxed, and the established result is shown to be sharp in a very strong sense.
\end{abstract}




\thispagestyle{empty}

\section{Introduction and main results}

Let $\H(\D)$\ denote the algebra of all analytic
functions in the unit disc $\D=\{z:|z|<1\}$ of the complex plane
$\mathbb{C}$. Let $\T$
be the boundary of $\D$. The \emph{Carleson square}  associated with an
interval $I\subset\T$ is the set $S(I)=\{re^{it}:\,e^{it}\in I,\,
1-|I|\le r<1\}$, where $|E|$ denotes the normalized Lebesgue measure of the
set $E\subset\T$. For our purposes it is also convenient to define
for each $a\in\D\setminus\{0\}$ the interval
$I_a=\left\{e^{i\t}:|\arg(a e^{-i\t})|\le \pi(1-|a|)\right\}$, and denote
$S(a)=S(I_a)$.
For $0<p\le\infty$, the
\emph{Hardy space} $H^p$ consists
of those $f\in\H(\D)$ for which
    \begin{equation*}\label{normi}
    \|f\|_{H^p}=\lim_{r\to1^-}M_p(r,f)<\infty,
    \end{equation*}
where
    $$
    M_p(r,f)=\left (\frac{1}{2\pi }\int_0^{2\pi}
    |f(re^{i\theta})|^p\,d\theta\right )^{\frac{1}{p}},\quad 0<p<\infty,
    $$
 and
    $$
    M_\infty(r,f)=\max_{0\le\theta\le2\pi}|f(re^{i\theta})|.
    $$
For the theory of the Hardy  spaces,
see~\cite{Duren1970,Garnett1981}.

For $0<p<\infty$ and $-1<\alpha<\infty$, the \emph{Dirichlet space} $\Dpa$ consists of those
$f\in \H(\mathbb D)$ such that
    \[
    \Vert f\Vert _{\Dpa}^p=\int_{\D}|f'(z)|^p(1-|z|^2)^\alpha\,dA(z)+|f(0)|^p<\infty,
    \]
where $dA(z)=\frac{dx\,dy}{\pi}$ is the normalized
Lebesgue area measure on $\D$.

The purpose of this study is to underscore operator theoretic differences between the closely related spaces $\Dp$ and $H^p$. Before going to that, it is appropriate to recall inclusion relations between these spaces. The classical Littlewood-Paley formula implies
$D^2_1=H^2$. Moreover, it is well known~\cite{Flett,LP} that 
    \begin{equation}\label{2}
    \Dp\subsetneq H^p,\quad0<p<2,
    \end{equation}
and
\begin{equation}\label{1}
H\sp p\subsetneq\Dp,\quad2<p<\infty.
\end{equation}
It is also worth mentioning that there are no inclusion relations between $\Dp$ and $\Dq$ when $p\neq q$~\cite{GP2}.

A natural way to illustrate differences between two given spaces is to consider classical operators acting on them. For example, if $0<p<2$, then the behavior of the \emph{composition operator} $C_\vp(f)=f\circ\vp$ reveals that $\Dp$ is in a sense a much smaller space than $H^p$.
Namely, it follows from Littlewood's subordination theorem that $C_\vp: H^p\to H^p$ is bounded for each $0<p<\infty$ and all analytic self-maps $\vp$ of $\D$, but in contrast to this, there are symbols~$\vp$ which induce unbounded operators $C_\vp:\Dp\to\Dp$ when $0<p<2$~\cite[Theorem~1.1(b)]{ChoKooSmithTams2003}. As in the case of Hardy spaces, any composition operator
$C_\vp: \Dp\to \Dp$ is bounded when $2\le p<\infty$.


There are operators which do not distinguish between $\Dp$ and $H^p$. For a given $g\in\H(\D)$, the \emph{generalized Hilbert operator}
$\hg$  is defined by
    \begin{equation}\label{H-g}
    \mathcal{H}_g(f)(z)=\int_0^1f(t)g'(tz)\,dt,
    \end{equation}
for any $f\in\H(\D)$ such that $\int_0^1|f(t)|\,dt<\infty$. If $1<p<\infty$, then $\hg:\Dp\to \Dp$ is bounded (compact) if and only if  $\hg:H^p\to \Dp$ is bounded (compact) by~\cite{GaGiPeSis}. Moreover, the same condition, depending on $g$ and $p$,
describes the boundedness (compactness) of the operators $\hg:\Dp\to H^p$ and $\hg: H^p\to H^p$ when $1<p\le 2$.
As far as we known, the problem of characterizing those symbols $g$ for which $\hg:\Dp\to H^p$ and $\hg: H^p\to H^p$ are bounded when $2<p<\infty$ remains unsolved.

We will next study operator theoretic differences between 
$\Dp$ and 
$H^p$ by considering the integral operator
    \begin{displaymath}
    T_g(f)(z)=\int_{0}^{z}f(\zeta)\,g'(\zeta)\,d\zeta,\quad
    z\in\D.
    \end{displaymath}
The bilinear operator $\left(f,g\right)\rightarrow \int fg'$  was introduced by 
Calder\'on in harmonic analysis in the $60$'s~\cite{Calderon65}. After his research on commutators of singular integral operators, this bilinear form and its different variations, usually called \lq\lq paraproducts\rq\rq, have been extensively studied and they have become a fundamental tool in harmonic analysis. Pommerenke was probably one of the first complex function theorists to consider the operator $T_g$. He used it in late 70's to study the space $\BMOA$, which
consists of those functions in the Hardy space~$H^1$ that have
\emph{bounded mean oscillation}
on the boundary $\T$~\cite{Pom}. The space
$\BMOA$ can be
equipped with several different equivalent norms \cite{Garnett1981}, here we
will use the one given by
    $$
    \|g\|^2_{\BMOA}=\sup_{a\in\D}\frac{\int_{S(a)}|g'(z)|^2(1-|z|^2)\,dA(z)}{1-|a|}+|g(0)|^2.
    $$
Two decades later, in late 90's, the pioneering works by Aleman and
Siskakis~\cite{AS0,AS} lead to an abundant research activity on the operator $T_g$. In particular, those analytic symbols~$g$ such that
$T_g:H^p\to H^q$ is bounded were characterized by Aleman, Cima and Siskakis~\cite{AC,AS0}. Their result in the case $p=q$ says that $T_g:H^p\to H^p$ is bounded if and only if $g\in\BMOA$.
Our first result shows that whenever $0<p\le2$, the domain space $H^p$ can be replaced by $\Dp$.

\begin{theorem}\label{th:1}
Let $0<p\le 2$ and $g\in\H(\D)$. Then the following are
equivalent:
\begin{itemize}
\item[\rm(i)] $T_g:\Dp\to H^p$ is bounded; \item[\rm(ii)] $T_g:
H^p\to H^p$ is bounded; \item[\rm(iii)] $g\in\BMOA$.
\end{itemize}
\end{theorem}

The implication (ii)$\Rightarrow$(i) is a direct consequence of \eqref{2}, so our contribution here consists of showing (i)$\Rightarrow$(iii).
The proof of the implication (ii)$\Rightarrow$(iii) in~\cite{AC,AS0} relies on several powerful properties of
$\BMOA$ and $H^p$ such as the conformal invariance
of $\BMOA$. Our proof is based on
a circle of ideas developed in \cite[Chapter~4]{PelRat}, and does not rely on these properties. Instead, the Fefferman-Stein formula~\cite{Zygmund59}, which states that
    \begin{equation}\label{fsest}
    \|f\|_{H^p}^p \asymp\int_\T S_f^p(\z)\,|d\z|+|f(0)|^p,
    \end{equation}
plays an important role in the reasoning. Here, $|d\z|$ denotes the arclength measure on $\T$, 
$S_f$ denotes the usual square function, also called the Lusin area function,
    \begin{equation}
    S_f(\zeta)=\left(\int_{\Gamma_\sigma(\zeta)}|f'(z)|^2\,dA(z)\right)^{1/2},\quad\zeta\in\T,
    \end{equation}
where $\Gamma_\sigma(\zeta)$ denotes a nontangential approach region (a Stolz angle) with vertex at
$\zeta$ and of aperture $\sigma$.

We also show that the statement in Theorem~\ref{th:1} drastically fails for $p>2$.
In order to give the precise statement, we will need to fix the notation. The \emph{disc algebra} $\mathcal{A}$ is the space of all analytic functions on $\D$ which are continuous on the boundary $\T$.
For $0<\alpha\le 1$, the \emph{Lipschitz space} $\Lambda(\alpha)$ consists of
those $g\in\H(\D)$, having a non-tangential limit $g(e^{i\theta})$
almost everywhere on $\T$, such that
    $$
    \sup_{\t\in[0,2\pi],\,0<t<1}\frac{|g(e^{i(\theta+t)})-g(e^{i\theta})|}{t^\alpha}<\infty.
    $$
The \lq\lq little oh\rq\rq counterpart of this space is denoted by $\lambda(\alpha)$.
The following chain of strict inclusions is known:
    $$
    \lambda(\alpha)\subsetneq\Lambda(\alpha)\subsetneq \mathcal{A}\subsetneq H^\infty\subsetneq\BMOA\subsetneq \B,\quad 0<\alpha\le 1.
    $$
Here, as usual, $\B$ stands for the \emph{Bloch space} which consists of
those $f\in\H(\D)$ such that
    $\|f\|_{\mathcal{B}}=\sup_{z\in\D}|f'(z)|(1-|z|^2)+|f(0)|<\infty.$

\begin{theorem}\label{th:2}
Let $2<p<\infty$ and $g\in\H(\D)$.
\begin{itemize}
\item[\rm(i)] If $T_g:\Dp\to H^p$ is bounded, then $g\in\BMOA$.
\item[\rm(ii)] There exist $g\in\mathcal{A}$ and $f\in\Dp$
such that $T_g(f)\notin H^p$.
\end{itemize}
\end{theorem}

Part~(ii) shows that $\Dp$ is in a sense a much larger space than $H^p$ when $p>2$. This is true because we may choose the inducing symbol $g$ to be as smooth as continuous on the boundary, but still a suitably chosen $f\in\Dp$
establishes $T_g(f)\notin H^p$. In contrast to this, when the inducing index of the domain space is strictly smaller than the one of the target space, that is $p<q$, then $T_g$ does not distinguish between $\Dp$ and $H^p$.

\begin{theorem}\label{th:3}
Let $0<p<q<\infty$ and $g\in\H(\D)$.
\begin{enumerate}
\item[\rm(a)] If $\frac1p-\frac1q\le1$, then the following
are equivalent:
\begin{enumerate}
\item[\rm(i)] $T_g:\Dp\to H^q$ is bounded; \item[\rm(ii)]
$T_g:H^p\to H^q$ is bounded; \item[\rm(iii)]
$g\in\Lambda(\frac1p-\frac1q)$.
\end{enumerate}
\item[\rm(b)] If $\frac1p-\frac1q>1$, then $T_g:\Dp\to H^q$ is
bounded if and only if $g$ is constant.
\end{enumerate}
\end{theorem}

Part (a) allows us to deduce a strengthened version of the classical result of Hardy-Littlewood which states that
a primitive of each function $f\in H^p$, $0<p<1$, belongs to $H^\frac{p}{1-p}$.

\begin{proposition}\label{factorization}
Let $p, \,p_1\, \mbox{and } \, p_2$ be positive numbers such that $p
<1 < p_2$ and $\frac{1}{p} = \frac{1}{p_1} + \frac{1}{p_2}$. If
$f\in\H(\D)$ such that $f = f_1 \cdot f_2$
where $f_1\in\mathcal{D}^{p_1}_{p_1-1}$ and  $f_2 \in \H(\D)$ satisfies $|f_2(z)| =\og
\left( \frac{1}{(1 - |z|)^{1/p_2}} \right)$, then $f$ is the
derivative of a function in $H^{\frac{p}{1 - p}}$.
\end{proposition}

The statement in Proposition~\ref{factorization} with $H^{p_1}$ in place of $\mathcal{D}^{p_1}_{p_1-1}$ was proved by
Aleman and Cima \cite[p.~158]{AC}. The strict inclusions \eqref{2} and \eqref{1} show that their result is better when $p_1<2$ meanwhile the situation is another way round when $p_1>2$.

An important ingredient in the proofs of both Theorems~\ref{th:1} and \ref{th:3} is the following result on a H\"ormander-type maximal function
    $$
    M(\vp)(z)=\sup_{I:\,z\in S(I)}\frac{1}{|I|}\int_{I}|\vp(\z)|\,\frac{|d\z|}{2\pi},\quad
    z\in\D,
    $$
defined for each $2\pi$-periodic function $\vp(e^{i\t})\in L^1(\T)$

\begin{lettertheorem}\label{co:maxbouhp}
Let $0<p\le q<\infty$ and $0<\alpha<\infty$ such that $p\alpha>1$.
Let $\mu$ be a positive Borel measure on $\D$. Then
$[M((\cdot)^{\frac{1}{\alpha}})]^{\alpha}:L^p(\T)\to L^q(\mu)$ is
bounded if and only if there exists a constant $C>0$ such that $\mu(S(I))\le C|I|^\frac{q}{p}$ for all $I\subset\T$
Moreover,
    $$
    \|[M((\cdot)^{\frac{1}{\alpha}})]^{\alpha}\|^q\asymp\sup_{I\subset\T}\frac{\mu\left(S(I) \right)}{|I|^\frac{q}p}.
    $$
\end{lettertheorem}

This result follows by the well-known works by Carleson~\cite{CarlesonL58,CarlesonL62}, and hence the measures $\mu$ for which $\mu(S(I))\le C|I|^\frac{q}{p}$ are known as \emph{$\frac{q}{p}$-Carleson measures}.
For more recent references, see either
\cite[Section~9.5]{Duren1970}, or the proof of \cite[Theorem~2.1]{PelRat} for a similar result. Theorem~\ref{co:maxbouhp} has been used to characterize so-called $q$-Carleson measures for Hardy spaces. Recall that, for a given Banach space (or a complete metric
space) $X$ of analytic functions on $\D$, a positive Borel measure
$\mu$ on $\D$ is called a \emph{$q$-Carleson measure for $X$}
if the identity operator $I_d:\, X\to L^q(\mu)$ is bounded. In nowadays these measures are a standard tool in the operator theory in spaces of analytic functions in $\D$.


Let us now turn back to the two remaining cases that are not covered by Theorems~\ref{th:1} and \ref{th:2}. They are the ones in which the operator $T_g$ acts from either $H^p$ or $\Dp$ to $\Dp$. It is easy to see that, in terms of the language of the previous paragraph, $T_g:H^p\to \Dq$ is bounded if and only if $\mu_{g,q}=|g'(z)|^q(1-|z|^2)^{q-1}\,dA(z)$
is a $q$-Carleson measure for $H^p$. Therefore, in this case the symbols $g$ that induce bounded operators get characterized by \cite[Theorem~9.5]{Duren1970}, when $q\ge p$, and \cite{Lu91} if $q<p$. Analogously, it follows that $T_g:\Dp\to \Dq$ bounded if and only if
$\mu_{g,q}$ is a $q$-Carleson measure for~$\Dp$. Unfortunately, as far as we know, the existing literature does not offer a characterization of these measures for the full range of parameter values in terms of a condition depending on $\mu$ only. It is known that they coincide with $q$-Carleson measures for $H^p$ and can therefore be described by the condition
    \begin{equation}
    \label{Carl.meas}
    \sup_{I\subset\T}\frac{\mu\left(S(I)\right)}{|I|^{q/p}}<\infty,
    \end{equation}
provided $q>p$~\cite[Theorem~1(a)]{GP:JFA06}. This statement remains valid also in the diagonal case $q=p$, if $p\le2$, but fails for $p>2$ \cite{GP:IE06,Wu}. In more general terms, the $p$-Carleson measures for $\Dpa$ are known excepting the case $\alpha=p-1$
for $p>2$~\cite{ARS,Wu}. This corresponds to the diagonal case $q=p>2$ which interests us in particular. What is known with respect to this case, is that $\mu$ being a 1-Carleson measure is a necessary but not a sufficient
condition for $\mu$ to be a $p$-Carleson measure for
$\mathcal{D}^p_{p-1}$~\cite{GP:IE06}, and that
the more restrictive condition
    $$
    \sup_{I\subset\T}\frac{\mu\left(S(I)\right)}{|I|\left(\log\frac{e}{|I|}\right)^{-p/2}}<\infty
    $$
is a sufficient condition for $I_d:\Dp\to L^p(\mu)$ to be bounded~\cite{GPP}. Our next result shows that this best known sufficient condition can be relaxed by one logarithmic term.

\begin{theorem}\label{th:cmdp}
Let $2<p<\infty$, and let $\mu$ be a positive Borel measure on
$\D$. If
    \begin{equation}\label{cardp}
    \sup_{I\subset\T}\frac{\mu\left(S(I)\right)}{|I|\left(\log\frac{e}{|I|}\right)^{-p/2+1}}<\infty,
    \end{equation}
then $\mu$ is a $p$-Carleson measure for $\Dp$.
\end{theorem}


We will see in Proposition~\ref{pr:sharp} that the statement in Theorem~\ref{th:cmdp} is sharp in a very strong sense.

The remaining part of the paper is organized as follows. In Section \ref{pre} we state and
prove some preliminary results. Theorems~\ref{th:1} and~\ref{th:3} and their expected analogues for compact operators as well as Proposition~\ref{factorization} are proved in Section~\ref{intoper}. In Section~\ref{sec:intmeans} we will deal with the growth of integral means of functions $f\in\Dp$, $p>2$, and we will prove Theorem~\ref{th:2}.

Before proceeding further, a word about notation to be used. We will write $\|T\|_{(X,Y)}$ for the norm of
an operator $T:X\to Y$, and if no confusion arises with regards to
$X$ and $Y$, we will simply write $\|T\|$. Moreover, for two real-valued functions $E_1,E_2$ we write $E_1\asymp
E_2$ or $E_1\lesssim E_2$, if there exists a positive constant $k$,
independent of the argument, such that $\frac{1}{k} E_1\leq E_2\leq k
E_1$ or $E_1\leq k E_2$, respectively.

\section{Preliminaries}\label{pre}

We begin with a straightforward but useful estimate that will be used in proofs of Theorems~\ref{th:1} and~\ref{th:3}.

\begin{lemma}\label{le:1}
Let $0<q,p<\infty$ and $g\in\H(\D)$. If $T_g:\Dp\to H^q$ is
bounded, then
    \begin{equation}\label{eq:nec1}
    M_\infty(r,g')\lesssim\frac{\|T_g\|_{(\Dp,
    H^q)}}{(1-r)^{1-\frac{1}{p}+\frac{1}{q}}},\quad 0\le r<1.
    \end{equation}
\end{lemma}

\begin{proof}
The functions
    $$
    F_{a,p,\g}(z)=\left(\frac{1-|a|^2}{1-\overline{a}z}\right)^{\frac{1+\g}{p}},\quad 0<\g<\infty,\quad
    a\in\D,
    $$
satisfy
    \begin{equation}\label{0}
    |F_{a,p,\g}(z)|\asymp 1,\quad z\in S(a),
    \end{equation}
and a calculation shows that
shows that
    $$
    \|F_{a,p,\g}\|^p_{\Dp}\asymp (1-|a|),\quad a\in\D.
    $$
Since $T_g:\Dp\to H^q$ is bounded by the assumption, the well
known relations $M_\infty(r,f)\lesssim
M_q\left(\frac{1+r}{2},f\right)(1-r)^{-\frac1q}$ and
$M_q(r,f')\lesssim M_q\left(\frac{1+r}{2},f\right)(1-r)^{-1}$,
valid for all $f\in\H(\D)$ (see \cite[Chapter $5$]{Duren1970}), yield
    \begin{equation*}
    \begin{split}
    |g'(a)|&=|(T_g(F_{a,p,\g}))'(a)|\lesssim \frac{M_q\left(\frac{1+|a|}{2},(T_g(F_{a,p,\g}))'\right)}{(1-|a|)^{\frac1q}}\\
    &\lesssim \frac{M_q\left(\frac{3+|a|}{4},T_g(F_{a,p,\g})\right)}{(1-|a|)^{1+\frac1q}}
    \lesssim \frac{\|T_g(F_{a,p,\g})\|_{H^q}}{(1-|a|)^{1+\frac1q}}\\
    &\lesssim \frac{\|T_g\|_{(\Dp,H^q)}\|F_{a,p,\g}\|_{\Dp}}{(1-|a|)^{1+\frac1q}}
    \lesssim \frac{\|T_g\|_{(\Dp,H^q)}}{(1-|a|)^{1+\frac1q-\frac{1}{p}}},\quad
    a\in\D,
    \end{split}
    \end{equation*}
and the assertion follows.
\end{proof}

We next recall some suitable reformulations of Lipschitz spaces $\Lambda (\alpha)$~\cite{Duren1970}.

\begin{letterlemma}\label{le:2}
Let $0<\alpha\le 1$. Then the following are equivalent:
\begin{itemize}
\item[\rm(i)] $g\in\Lambda (\alpha)$;
\item[\rm(ii)] $M_\infty(r,g')=\og\left(\frac{1}{(1-r)^{1-\a}}\right),\quad r\to 1^-$;
\item[\rm(iii)] $d\mu_g(z)=|g'(z)|^2(1-|z|^2)\,dA(z)$ satisfies
        $\displaystyle
        \sup_{I\subset \T}\frac{\mu_g\left(S(I)\right)}{|I|^{2\a+1}}<\infty.
        $
\end{itemize}
\end{letterlemma}

We will also need the following result~\cite[Theorem~1(i)]{GP:JFA06}.

\begin{lettertheorem}\label{pr:dpq}
Let $0<p<q<\infty$ and $\mu$ be a positive Borel measure on~$\D$.
Then $\mu$ is a $q$-Carleson measure for $\Dp$ if and only if
$\mu$ is a $\frac{q}{p}$-Carleson measure.
Moreover,
    $$
    \|Id_{\left(\Dp,L^q(\mu)\right)}\|^q\asymp\sup_{I\subset\T}\frac{\mu\left(S(I) \right)}{|I|^\frac{q}p}.
    $$
\end{lettertheorem}

\section{Integral operators from Hardy to Dirichlet type spaces $\Dq$}\label{intoper}

\begin{Prf}{\em{Theorem \ref{th:1}}.}
It is known that $T_g: H^p\to H^p$ is bounded if and only if
$g\in\BMOA$~\cite{AC}, and therefore (ii) and (iii) are equivalent. Moreover, since $\Dp\subset H^p$ for $0<p\le 2$, (ii)
implies (i). To complete the proof we will show that $g\in\BMOA$,
whenever $T_g: \Dp\to H^p$ is bounded. To see this, note first
that $\|g\|_{\B}\lesssim\|T_g\|_{(\Dp,H^p)}$ by Lemma~\ref{le:1},
and thus $g\in\B$. Let now $1<\a,\b<\infty$ such that
$\b/\a=p/2<1$, and let $\a'$ and $\b'$ be the conjugate indexes of
$\a$ and $\b$. Assume for a moment that $g'$ is continuous
on~$\overline{\D}$. Then \eqref{0}, Fubini's theorem and
H\"older's inequality yield
    \begin{equation}
    \begin{split}\label{eq:tgb1hp}
    &\int_{S(a)}|g'(z)|^2(1-|z|^2)\,dA(z)\\
    &\asymp\int_\T\left(\int_{S(a)\cap\Gamma_\sigma(\z)}|g'(z)|^2|F_{a,p,\g}(z)|^2\,dA(z)\right)^{\frac1\a+\frac1{\a'}}\,|d\z|\\
    &\le\left(\int_\T\left(\int_{\Gamma_\sigma(\z)}|g'(z)|^2|F_{a,p,\g}(z)|^2\,dA(z)\right)^\frac{\b}{\a}\,|d\z|\right)^\frac1\b\\
    &\quad\cdot\left(\int_\T\left(\int_{\Gamma_\sigma(\zeta)\cap
    S(a)}|g'(z)|^2\,dA(z)\right)^\frac{\b'}{\a'}\,|d\z|\right)^\frac1{\b'}\\
    &\asymp\|T_g(F_{a,p,\g})\|_{H^p}^\frac{p}{\b}\|S_g(\chi_{S(a)})\|_{L^\frac{\b'}{\a'}(\T)}^\frac{1}{\a'},\quad a\in\D,
    \end{split}
    \end{equation}
where
    $$
    S_g(\varphi)(\zeta)=\int_{\Gamma_\sigma(\zeta)}|\varphi(z)|^2|g'(z)|^2\,dA(z),\quad
    \zeta\in\T,
    $$
for any bounded function $\varphi$ on $\D$. Now
$\left(\frac{\b'}{\a'}\right)'=\frac{\b(\a-1)}{\a-\b}>1$, and
hence
    \begin{equation}
    \begin{split}\label{eq:tgb2hp}
    \|S_g(\chi_{S(a)})\|_{L^\frac{\b'}{\a'}(\T)}
    =\sup_{\|h\|_{L^{\frac{\b(\a-1)}{\a-\b}}(\T)}\le1}
    \left|\int_\T h(\z)S_g(\chi_{S(a)})(\z)\,|d\z|\right|
    \end{split}
    \end{equation}
by the duality. To estimate the right hand side, we shall write $I(z)$ for the arc $\left\{\zeta\in\T: z \in \Gamma_\sigma(\z)\right\}$ with
$|I(z)|\asymp 1-|z|$. Then Fubini's
theorem, H\"older's inequality and Theorem~\ref{co:maxbouhp}
yield
     \begin{equation}
     \begin{split}\label{eq:tgb3hp}
    &\left|\int_\T h(\z)S_g(\chi_{S(a)})(\z)\,|d\z|\right|
    \le\int_\T|h(\z)|\int_{\Gamma_\sigma(\zeta)\cap
    S(a)}|g'(z)|^2\,dA(z)\,|d\z|\\
    &\asymp \int_{S(a)}|g'(z)|^2(1-|z|^2)\left(\frac{1}{1-|z|^2}\int_{I(z)}|h(\z)|\,|d\z|\right)\,dA(z)\\
    &\lesssim\int_{S(a)}|g'(z)|^2(1-|z|^2)M(|h|)(z)\,dA(z)\\
    &\le\left(\int_{S(a)}|g'(z)|^2(1-|z|^2)\,dA(z)\right)^\frac{\a'}{\b'}\\
    &\quad\cdot\left(\int_\D
    M(|h|)^{\left(\frac{\b'}{\a'}\right)'}|g'(z)|^2(1-|z|^2)\,dA(z)\right)^{1-\frac{\a'}{\b'}}\\
    &\lesssim\left(\int_{S(a)}|g'(z)|^2(1-|z|^2)\,dA(z)\right)^\frac{\a'}{\b'}\\
    &\quad\cdot\left(\sup_{a\in\D}\frac{\int_{S(a)}|g'(z)|^2(1-|z|^2)\,dA(z)}{1-|a|}\right)^{1-\frac{\a'}{\b'}}
    \|h\|_{L^{\left(\frac{\b'}{\a'}\right)'}(\T)}.
    \end{split}
    \end{equation}
Since any dilated function $g_r(z)=g(rz)$, $0<r<1$, is analytic on $D\left(0,\frac{1}{r}\right)$, by replacing $g$ by $g_r$ in \eqref{eq:tgb1hp}--\eqref{eq:tgb3hp},
we deduce
    \begin{equation}\label{59}
    \begin{split}
    &\int_{S(a)}|g_r'(z)|^2(1-|z|^2)\,dA(z)\\
    &\lesssim\|T_{g_r}(F_{a,p,\g})\|_{H^p}^\frac{p}{\b}
    \left(\int_{S(a)}|g_r'(z)|^2(1-|z|^2)\,dA(z)\right)^\frac{1}{\b'}\\
    &\quad\cdot\left(\sup_{a\in\D}\frac{\int_{S(a)}|g_r'(z)|^2(1-|z|^2)\,dA(z)}{1-|a|}\right)^{\frac{1}{\a'}\left(1-\frac{\a'}{\b'}\right)}.
    \end{split}
    \end{equation}
We claim that there exists $\g$ and a constant $C=C(p,\g)>0$ such
that
    \begin{equation}\label{2n}
    \sup_{0<r<1}\|T_{g_r}(F_{a,p,\g})\|_{H^p}^p\le
    C\|T_{g}\|_{(\Dp,H^p)}^p(1-|a|),\quad a\in\D,
    \end{equation}
the proof of which is postponed for a moment. Now this combined
with \eqref{59} and Fatou's lemma yield
    $$
    \sup_{a\in\D}\frac{\int_{S(a)}|g'(z)|^2(1-|z|^2)\,dA(z)}{1-|a|}\lesssim\|T_{g}\|_{(\mathcal{D}^p_{p-1},H^p)}^2,
    $$
and so $g\in\BMOA$.

It remains to prove \eqref{2n}. To see this fix $\g>p$. Recall that
    $$
    \|T_{g_r}(F_{a,p,\g})\|^p_{H^p}
    \asymp  \int_{\T}\left( \int_{\Gamma_\sigma(\zeta)} r^2 |g'(rz))|^2 |F_{a,p,\g}(z)|^2dA(z)\right)^{p/2}
    |d\zeta|.
    $$
If $|a| < \frac{1}{2}$, then
    \begin{equation*}
    \begin{split}
    \|T_{g_r}(F_{a,p,\g})\|^p_{H^p}
    &\lesssim(1 -|a|)^{\gamma + 1} \int_{\T}\left(\int_{\Gamma_\sigma(\zeta)}
    r^2 |g'(rz)|^2 dA(z)\right)^{\frac{p}{2}}|d\zeta|  \\
    & \asymp(1 -|a|)^{\gamma + 1}\|g_r-g(0)\|_{H^p}^p
    \le (1 -|a|)\|g-g(0)\|_{H^p}^p \\
    & =(1 -|a|)\|T_g(1)\|_{H^p}^p \lesssim(1 -|a|)\|T_g\|_{(\Dp,
    H^p)}^p.
    \end{split}
    \end{equation*}
Let now $\frac{1}{2} \le |a| < \frac{1}{2 - r}$. Then
$|1-\overline{a}rz|\le2|1-\overline{a}z|$ for all $z\in\D$, and
hence
    \begin{equation*}
    \begin{split}
    \|T_{g_r}(F_{a,p,\g})\|^p_{H^p}
    &\lesssim\|(T_{g}(F_{a,p,\g}))_r\|^p_{H^p}
    \le\|T_g(F_{a,p,\g}) \|_{H^p}^p \\
    & \le \|T_g \|_{(\Dp, H^p)}^p\|F_{a,p,\g}\|^p_{\Dp}
    \asymp \|T_g \|_{(\Dp, H^p)}^p(1 - |a|).
    \end{split}
    \end{equation*}
In the remaining case $\frac{1}{2-r}\le|a|<1$ we have
$r\le2-\frac{1}{|a|}\le|a|$. Now $\g>p$, and hence
    \begin{equation*}
    \begin{split}
    \|T_{g_r}(F_a)\|^p_{H^p}
    & \lesssim M^p_{\infty}(r, g')(1 - |a|)^{\gamma + 1}\int_{\T}\left( \int_{\Gamma_\sigma(\zeta)}
    \frac{dA(z)}{|1 - \overline{a}z|^{\frac{2(\gamma + 1)}{p }}} \right)^{p/2} |d\zeta|  \\
    & \lesssim M^p_{\infty}(|a|, g')
    (1 - |a|)^{\gamma + 1} \left \|  \frac{1}{(1 - \overline{a}z)^{\frac{\gamma + 1}{p } - 1}}\right  \|_{H^p}^p  \\
    & \asymp \left( M_{\infty}(|a|, g')(1 - |a|)\right)^p (1 - |a|)
    \le \|g\|_{\B}^p (1-  |a|)  \\
    & \lesssim \|T_g\|^p_{(\Dp, H^p)}(1-  |a|).
    \end{split}
    \end{equation*}
By combining these three separate cases we deduce \eqref{2n}.
\end{Prf}

\medskip

Next, we will prove Theorem~\ref{th:3} by using similar ideas that were employed in the proof of Theorem~\ref{th:1}.

\medskip

\begin{Prf}{\em{Theorem \ref{th:3}}.\,}
It is known that (ii) and (iii) are equivalent~\cite{AC}. Further,
 Lemma~\ref{le:1} and Lemma \ref{le:2} give (i)$\Rightarrow$(iii) and (b).
Moreover, if $0<p\le 2$, then $\mathcal{D}^p_{p-1}\subset H^p$
and hence, in this case, (ii) implies (i). To complete the proof,
we show that (iii) implies (i) when $2<p<\infty$. Since $q>2$,
$L^{q/2}(\T)$ can be identified with the dual of
$L^{\frac{q}{q-2}}(\T)$, that is,
$L^{q/2}(\T)=\left(L^{\frac{q}{q-2}}(\T)\right)^\star$. Therefore,
$T_g:\Dp\to H^q$ is bounded if and only if
    \begin{equation*}
    \left|
    \int_\T\,h(\zeta)\left(\int_{\Gamma_\sigma(\zeta)}|f(z)|^2|g'(z)|^2\,dA(z)\right)\,|d\zeta|
    \right|\lesssim\|h\|_{L^{\frac{q}{q-2}}(\T)}\|f\|^2_{\Dp}
    \end{equation*}
for all $h\in L^{\frac{q}{q-2}}(\T)$ and
$f\in\mathcal{D}^p_{p-1}$. To see this, we use first Fubini's
theorem to obtain
    \begin{equation*}
    \begin{split}
    &\left|
    \int_\T\,h(\zeta)\left(\int_{\Gamma_\sigma(\zeta)}|f(z)|^2|g'(z)|^2\,dA(z)\right)\,d|\zeta|\right|\\
    &\le\int_\D|f(z)|^2|g'(z)|^2\left(\int_{I(z)}|h(\zeta)|\,|d\zeta|\right)\,dA(z)\\
    &\lesssim\int_\D|f(z)|^2M(|h|)(z)|g'(z)|^2(1-|z|^2)\,dA(z).
    \end{split}
    \end{equation*}
Since $|g'(z)|^2(1-|z|^2)\,dA(z)$ is a
$\left(2\left(\frac1p-\frac1q\right)+1\right)$-Carleson measure by
Lemma~\ref{le:2}, and $2(\frac1p-\frac1q)+1=(2+p-\frac{2p}{q})/p$,
we may estimate the last integral upwards by H\"{o}lder's
inequality, Theorem~\ref{pr:dpq} and Theorem~\ref{co:maxbouhp} to
    \begin{equation*}
    \begin{split}
    &\left(\int_\D|f(z)|^{2+p-\frac{2p}{q}}|g'(z)|^2(1-|z|^2)\,dA(z)\right)^{\frac{2q}{(2+p)q-2p}}\\
    &\quad\cdot\left(\int_\D\left(M(|h|)(z)\right)^{1+\frac{2q}{p(q-2)}}|g'(z)|^2(1-|z|^2)\,dA(z)\right)^{\frac{1}{1+\frac{2q}{p(q-2)}}}\\
    &\lesssim\|f\|^2_{\Dp}\|h\|_{L^{\frac{q}{q-2}}(\T)}.
    \end{split}
    \end{equation*}
These estimates give the desired inequality for all $h\in
L^{\frac{q}{q-2}}(\T)$ and $f\in\mathcal{D}^p_{p-1}$, and thus
$T_g:\Dp\to H^q$ is bounded.
\end{Prf}

\medskip

We now prove Proposition~\ref{factorization}.

\medskip

\begin{Prf}
{\em{Proposition \ref{factorization}}}.\,
Let $F_2$ be such that  $F_2' = f_2$. Then $|F_2'(z)| = \og\left(
\frac{1}{(1 - |z|)^{1/p_2}} \right)$ by the assumption, and hence $F_2 \in
\Lambda(1 - \frac{1}{p_2})$ by Lemma~\ref{le:2}. Now Theorem~\ref{th:3}
implies that the integral operator $T_{F_2}:\mathcal{D}^{p_1}_{p_1-1}\to H^{\frac{p}{1 - p}}$ is bounded, and since $f_1\in\mathcal{D}^{p_1}_{p_1-1}$ by the assumption, we deduce
    $
    T_{F_2}(f_1)(z)=\int_0^z F'_2(\zeta)f_1(\zeta)\,d\zeta
    =\int_0^z f(\zeta)\,d\zeta\in H^{\frac{p}{1-p}},
    $
which gives  the assertion.
\end{Prf}

\medskip

We finish this section by proving the expected versions of Theorems~\ref{th:1} and~\ref{th:3} for compact operators. The next auxiliary result is standard, and therefore its proof is omitted.

\begin{lemma}\label{le:compacidadtg}
Let $0<p,q<\infty$ and $g\in \H(\D)$. Then the following are equivalent:
\begin{enumerate}
\item[\rm(i)]
$T_g:\Dp\to H^q$ is compact;
\item[\rm(ii)]
For any sequence of analytic functions $\{f_n\}_{n=1}^\infty$ on $\D$ that converges uniformly to $0$ on compact subsets of $\D$ and satisfies
$\sup_{n\in\N}\|f_n\|_{\Dp}<\infty$, we have $\lim_{n\to\infty}\|T_g(f_n)\|_{H^q}=0$.
\end{enumerate}
\end{lemma}

Obviously the statement in this lemma remains valid if $H^p$ is replaced by~$\Dp$.

The space $\VMOA$
consists of those functions in the Hardy space $H^1$ that have
\emph{vanishing mean oscillation} on the
boundary~$\T$. It is known that this space is the closure of polynomials in $\BMOA$ and is characterized by the condition
    $$
    \lim_{|a|\to 1^-}\frac{\int_{S(a)}|g'(z)|^2(1-|z|^2)\,dA(z)}{1-|a|}=0.
    $$

\begin{theorem}\label{th:1-com}
Let $0 <  p\le 2$ and $g\in\H(\D)$. Then the following are
equivalent:
\begin{itemize}
\item[\rm(i)] $T_g: \Dp\to H^p$ is compact;
 \item[\rm(ii)] $T_g: H^p \to H^p$ is compact;
 \item[\rm(iii)] $g\in\VMOA$.
\end{itemize}
\end{theorem}

\begin{proof}
It is known that (ii) and (iii) are equivalent by~\cite{AC}. Moreover, by bearing in mind Lemma~\ref{le:compacidadtg} and \eqref{2},
we see that (ii) implies (i). It remains to show that $g\in\VMOA$,
whenever $T_g: \Dp\to H^p$ is compact. Since the proof of this implication is similar to its counterpart in the proof of Theorem~\ref{th:1}, we only show in detail those steps that are significantly different. First observe, that $g\in\BMOA$ by Theorem~\ref{th:1}. Let $f_{a,p,\g}=\frac{F_{a,p,\g}}{(1-|a|)^{1/p}}$, where $\gamma>0$ and $F_{a,p,\g}$ are those functions defined
in the proof of Lemma~\ref{le:1}. It is clear that $\|f_{a,p,\g}\|_{\Dp}\asymp1$
and $f_{a,p}\to 0$, as $|a|\to 1^-$, uniformly in compact subsets
of $\D$. Therefore $\|T_g(f_{a,p,\g})\|_{H^p}\to 0$, as
$|a|\to 1^-$, by Lemma~\ref{le:compacidadtg}. Now, let
$1<\a,\b<\infty$ such that $\b/\a=p/2<1$. Arguing as in
\eqref{eq:tgb1hp}, we deduce
    \begin{equation*}
    \frac{1}{(1-|a|)^\frac2p}\int_{S(a)}|g'(z)|^2(1-|z|^2)\,dA(z)
    \lesssim\|T_g(f_{a,p,\gamma})\|_{H^p}^\frac{p}{\b}
    \|S_g(\chi_{S(a)}f_{a,p,\gamma})\|_{L^{\frac{\b'}{\a'}}(\T)}^\frac{1}{\a'}
    \end{equation*}
for all $a\in\D$. Following the reasoning in the proof of
Theorem~\ref{th:1} and bearing in mind that $g\in\BMOA$, we obtain
    \begin{equation*}\index{$f_{a,p}$}
    \frac{\int_{S(a)}|g'(z)|^2(1-|z|^2)\,dA(z)}{(1-|a|)^\frac{2}{p}}
    \lesssim\|T_g(f_{a,p,\g})\|_{H^p}^\frac{p}{\b}
    \frac{\left(\int_{S(a)}|g'(z)|^2(1-|z|^2)\,dA(z)\right)^{\frac{\a'}{\b'}\cdot\frac{1}{\a'}}}
    {(1-|a|)^{\frac{2}{p}\cdot\frac{1}{\a'}}},
    \end{equation*}
which is equivalent to
    $$
    \frac{\int_{S(a)}|g'(z)|^2(1-|z|^2)\,dA(z)}{(1-|a|)}
    \lesssim\|T_g(f_{a,p,\g})\|_{H^p}^p.$$
Therefore $g\in\VMOA$.
\end{proof}

It is known that the \lq\lq little oh\rq\rq analogue of Lemma~\ref{le:2} is valid. This together with
appropriate modifications in the proofs of Lemma~\ref{le:1} and Theorem~\ref{th:3} give the next result.

\begin{theorem}\label{th:3-compact}
Let $0 <  p < q < \infty$,   $\frac1p-\frac1q\le1$,  and  $g\in\H(\D)$.
The following
are equivalent:
\begin{enumerate}
\item[\rm(i)] $T_g:\Dp\to H^q$ is compact;
 \item[\rm(ii)]
$T_g:H^p\to H^q$ is compact;
 \item[\rm(iii)]
$g\in\lambda(\frac1p-\frac1q)$.
\end{enumerate}
\end{theorem}

\section{Growth of integral means of functions in $\Dp$.}\label{sec:intmeans}

In this section we will prove sharp estimates for the growth of $M_p(r,f)$ when $f\in\Dp$ and $2<p<\infty$. If $f\in\Dp$ and $0<p<2$, then $M_p(r,f)$ is uniformly bounded due to \eqref{2}.

\begin{lemma}\label{integralmeans}
Let $2<p<\infty$ and $\Phi:[0,1)\to (1,\infty)$ be a
differentiable increasing unbounded function such that
$\frac{\Phi'(r)}{\Phi(r)}(1-r)$ is decreasing. Then the following hold:
\begin{itemize}
\item[\rm(i)] For any $f\in\Dp$, $M_p(r,f)= \op\left(\left(\log\frac{e}{1-r}\right)^{\frac{1}{2}-\frac{1}{p}}\right),\quad r\to 1^-$;
 \item[\rm(ii)] there exists $f\in\Dp$ such that
    \begin{equation}\label{sharp}
    M_q(r,f)\gtrsim
    \left(\log\frac{e}{1-r}\right)^{\frac{1}{2}}\left(\frac{\Phi'(r)}{\Phi^2(r)}(1-r)\right)^{\frac{1}{p}},\quad 0<r<1,
\end{equation}
for any fixed $0<q<\infty$.
\end{itemize}
 \end{lemma}

Part~(i) is essentially known, but we include a proof for the sake of completeness. Part (ii), apart from showing that (i) is sharp in a very strong, will be used to prove
Theorem~\ref{th:2}(ii) and the sharpness of Theorem~\ref{th:cmdp}.
It is also worth noticing that each function
    \begin{equation}\label{Eq:PesosEnV-alpha}
    \Phi_{N,\a}(r)=\left(\log_{N}\frac{\exp_{N}2}{1-r}\right)^\a,\quad N\in\N=\{1,2,\ldots\},\quad0<\a<\infty,
    \end{equation}
satisfies both hypotheses on the auxiliary function $\Phi$ in Lemma~\ref{integralmeans}. Here, as usual,
$\log_nx=\log(\log_{n-1}x)$, $\log_1x=\log x$,
$\exp_n x=\exp(\exp_{n-1}x)$ and $\exp_1x=e^x$.

\medskip

\begin{Prf}{\em{Lemma~\ref{integralmeans}}.}\,
(i) First observe that \cite[Theorem~1.4]{GPP} yields
    \begin{equation}\label{dpvp2}
    \Dp\subset A^p_{v_{\frac{p}{2}}},\quad\|f\|^p_{\Dp} \gtrsim \|f\|^p_{A^p_{v_{\frac{p}{2}}}},\quad f\in\H(\D),
    \end{equation}
where $A^p_{v_{\frac{p}{2}}}$ denotes the weighted Bergman space induced by the rapidly increasing weight
    $
    v_{\frac{p}{2}}(z)=\frac{1}{(1-|z|)\left(\log\frac{e}{1-|z|}\right)^{\frac{p}{2}}},\quad z\in\D,
    $
see \cite[Section~1.2]{PelRat}. Therefore,
    \begin{equation*}\begin{split}
    \|f\|^p_{\Dp}&\gtrsim \|f\|^p_{A^p_{v_{\frac{p}{2}}}}
    \ge \int_r^1 sM^p_p(s,f)v_{\frac{p}{2}}(s)\,ds
     \ge M^p_p(r,f)\int_r^1 s v_{\frac{p}{2}}(s)\,ds\\
    &\asymp M_p^p(r,f)\left(\log\frac{e}{1-r}\right)^{1-\frac{p}{2}},\quad 0<r<1,
    \end{split}\end{equation*}
and (i) follows.

(ii) Let $\Phi$ be as in the lemma. Consider the lacunary series
    \begin{equation}\label{eqlac}
    f(z)=\sum_{k=1}^\infty\left(\frac{h(r_k)-h(r_{k-1})}{\Phi(r_k)}\right)^\frac{1}{p}
    z^{2^k},\quad r_k=1-2^{-k},\quad k\in\N,
    \end{equation}
where $h(r)=\log\Phi(r)$ is a positive function such that
$h'(r)(1-r)$ is decreasing by the assumptions.
By \cite[Proposition~3.2]{GP:IE06},
    \begin{equation*}
    \begin{split}
    \|f\|_{\mathcal{D}^p_{p-1}}^p&\lesssim\sum_{k=1}^\infty\left(\frac{h(r_k)-h(r_{k-1})}{\Phi(r_k)}\right)\\
    &=\sum_{k=1}^\infty\frac{\int_{r_{k-1}}^{r_k}h'(t)\,dt}{\Phi(r_k)}\le\int_0^1\frac{h'(t)}{\Phi(t)}\,dt=\Phi(0)^{-1}<1,
    \end{split}
    \end{equation*}
and thus $f\in\Dp$.

On the other hand
    \begin{equation*}
    \begin{split}
    M_2^2(r_N,f)&=\sum_{k=1}^\infty \left(\frac{h(r_k)-h(r_{k-1})}{\Phi(r_k)}\right)^\frac{2}{p}
    r_N^{2^{k+1}}\\
    &\ge\sum_{k=1}^N \left(\frac{h(r_k)-h(r_{k-1})}{\Phi(r_k)}\right)^\frac{2}{p}
    r_N^{2^{k+1}}\\
    &\ge\frac{r_N^{2^{N+1}}}{(\Phi(r_N))^\frac2p}
    \sum_{k=1}^N\left(\int_{r_{k-1}}^{r_k}h'(s)(1-s)\frac{ds}{1-s}\right)^\frac2p\\
    &\ge\frac{r_N^{2^{N+1}}(\log2)^\frac2p}{(\Phi(r_N))^\frac2p}\sum_{k=1}^N\left(h'(r_k)(1-r_k)\right)^\frac2p\\
    &\gtrsim\frac{1}{(\Phi(r_N))^\frac2p}\left(h'(r_N)(1-r_N)\right)^\frac2pN
    \end{split}
    \end{equation*}
Let $r\in[\frac12,1)$ be given, and choose $N\in\N$ such that $r_N\le
r<r_{N+1}$. Then \cite[Theorem $8.20$ in p. $215$ Vol I]{Zygmund59} yields
    \begin{equation*}
    \begin{split}
    M_q^2(r,f)&\asymp M_2^2(r,f)\ge
    M_2^2(r_N,f)\gtrsim\frac{1}{(\Phi(r_N))^\frac2p}\left(h'(r_N)(1-r_N)\right)^\frac2pN\\
    &\gtrsim\frac{1}{(\Phi(r))^\frac2p}\left(h'(r)(1-r)\right)^\frac2p\log\frac{e}{1-r}\\
    &\asymp
     \left(\log\frac{e}{1-r}\right)\left(\frac{\Phi'(r)}{\Phi^2(r)}(1-r)\right)^{\frac{2}{p}},
    \end{split}
    \end{equation*}
    which finishes the proof.
\end{Prf}

\medskip

With these preparations we are ready to prove Theorem~\ref{th:2}.

\medskip

\begin{Prf}{\em{Theorem \ref{th:2}}}
(i) If $T_g: \Dp\to H^p$ is bounded, then $T_g: H^p\to H^p$ is
bounded because $H\sp p\subsetneq\Dp$ for $2<p<\infty$ by \eqref{1}, and hence
$g\in\BMOA$.

(ii) 
In this part we use ideas from the proof of \cite[Theorem~2.1]{GP:IE06}.
Take a function $\Phi$ as in Lemma~\ref{integralmeans} and let $f\in\Dp$ the lacunary series associated to $\Phi$ via \eqref{eqlac}.
By using \cite[Theorem~8.25, Chap. $V$,
Vol.~I]{Zygmund59}, we find two constants $A>0$ and $B>0$ such
that for every $r\in (0,1)$ the set
    \begin{equation}\label{eq:f2}
    E_r=\{t\in[0,2\pi]:|f(re^{it})|>BM_2(r,f)\}
    \end{equation}
has the Lebesgue measure greater than or equal to $A$. Let now $g$
be a lacunary series. By using \cite[Lemma~6.5, Chap.~V,
Vol.~I]{Zygmund59} we find a constant $C_1>0$ such that
    \begin{equation}\label{eq3}
    \int_{E_r}|g'(re^{it})|^2\,dt \ge C_1AM^2_2(r,g')=C_2M^2_2(r,g'),\quad
    0<r<1,
    \end{equation}
where $C_2=C_1A$. Bearing in mind the definition \eqref{eq:f2} of
the sets $E_r$ and using \eqref{eq3},
we obtain
    \begin{equation}\label{eq:f}
    \begin{split}
   &  \|T_g(f)\|^2_{H^p} \ge  \|T_g(f)\|^2_{H^2}
     \gtrsim \int_{\D}|f(z)|^2|g'(z)|^2(1-|z|^2)\,dA(z)
    \\ &\ge
    \int_{0}^1r(1-r)\int_{E_r}|f(re^{it})|^2|g'(re^{it})|^2\,dt\,dr
    \\ & \ge
    B^2\int_{0}^1r(1-r)M_2^2(r,f)\int_{E_r}|g'(re^{it})|^2\,dt\,dr
    \\ & \ge
    B^2C_2\int_{0}^1r(1-r)M_2^2(r,f)M^2_2(r,g')\,dr
    \\ & \ge
    B^2C_2C\int_{0}^1r(1-r)\left(\log\frac{e}{1-r}\right)\left(\frac{\Phi'(r)}{\Phi^2(r)}(1-r)\right)^{\frac{2}{p}}
    M^2_2(r,g')\,dr.
    \end{split}
    \end{equation}
Choose now $\Phi(r)=\left(\log\frac{e}{1-r}\right)^{\ep}$, where $0<\ep<\frac{p}{2}-1$, so that
    $$
    \left(\log\frac{e}{1-r}\right)\left(\frac{\Phi'(r)}{\Phi^2(r)}(1-r)\right)^{\frac{2}{p}}
    \asymp\left(\log\frac{e}{1-r}\right)^{1-\frac{2}{p}(1+\ep)}.
    $$
Further, let
    $$
    g(z)=\sum_{j=0}^\infty \frac{1}{(j+1)\left(\log{j+1}\right)^\alpha} z^{2^{2^j}},\quad1<\alpha<\infty.
    $$
Then clearly $g\in \mathcal{A}$. 
Moreover, since $\om(r)=(1-r)\left(\log\frac{e}{1-r}\right)^{1-\frac{2}{p}(1+\ep)}$ is a so-called regular weight, we deduce
    $$
    \int_{0}^1r^{2n+1}\om(r)\asymp n^{-1}\om(1-n^{-1}),\quad n\in\N,
    $$
by \cite[Lemma~1.3 and (1.1)]{PelRat}. This together with \eqref{eq:f} yields
\begin{equation*}
    \begin{split}
   &  \|T_g(f)\|^2_{H^p} \gtrsim
    \int_{0}^1r(1-r)\left(\log\frac{e}{1-r}\right)^{1-\frac{2}{p}(1+\ep)}M^2_2(r,g')\,dr
    \\ & \asymp \sum_{j=1}^\infty \frac{2^{2^{j+1}}}{(j+1)^2\left(\log{j+1}\right)^{2\alpha}} \left(\int_{0}^1r^{2^{2^j+1}-1}(1-r)\left(\log\frac{e}{1-r}\right)^{1-\frac{2}{p}(1+\ep)}\,dr\right)
     \\ & \asymp \sum_{j=1}^\infty \frac{2^{(j+1)\left(1-\frac{2}{p}(1+\ep)\right)}}{(j+1)^2\left(\log{j+1}\right)^{2\alpha}}=\infty,
    \end{split}
    \end{equation*}
and finishes the proof.
\end{Prf}

\section{Carleson measures for the Dirichlet space
$\mathcal{D}^p_{p-1}$}\label{Section:Dp}

The statement in Theorem~\ref{th:cmdp} follows directly by \eqref{dpvp2} and \cite[Theorem~2.1]{PelRat} with
$\om=v_{p/2}$. We next show that this result is sharp in a very strong sense. For this purpose, the following lemma is needed.

\begin{lemma}\label{Lem:BernoulliHospital}
Let $2<p<\infty$, and let $\Phi:[0,1)\to(0,\infty)$ be a
differentiable increasing function such that
    \begin{equation}\label{111}
    \frac{\Phi(r)}{\left(\log\frac{e}{1-r}\right)^{\frac{p}{2}-1}}\to0,\quad
    r\to1^-,
    \end{equation}
and
    \begin{equation}\label{112}
    m=-\liminf_{r\to1^-}\frac{\Phi'(r)}{\Phi(r)}(1-r)\log\frac{e}{1-r}>1-\frac{p}{2}.
    \end{equation}
Then
    $$
    \int_r^1\frac{\Phi(s)\,ds}{(1-s)\left(\log\frac{e}{1-s}\right)^{\frac{p}{2}}}
    \lesssim\frac{\Phi(r)}{\left(\log\frac{e}{1-r}\right)^{\frac{p}{2}-1}},\quad
    r\in(0,1).
    $$
\end{lemma}

\begin{proof}
By the Bernouilli-l'H\^{o}pital theorem,
    \begin{equation*}
    \begin{split}
    \limsup_{r\to1^-}\frac{\int_r^1\frac{\Phi(s)\,ds}{(1-s)\left(\log\frac{e}{1-r}\right)^{\frac{p}{2}}}}
    {\frac{\Phi(r)}{\left(\log\frac{e}{1-r}\right)^{\frac{p}{2}-1}}}
    \le\left(m+\frac{p}{2}-1\right)^{-1}\in(0,\infty),
    \end{split}
    \end{equation*}
and the assertion follows.
\end{proof}

If
$\Phi_c(r)=\left(\log\frac{e}{1-r}\right)^c$ and $c>0$, then
    $$
    \frac{\Phi_c'(r)}{\Phi_c(r)}(1-r)\log\frac{e}{1-r}=c,\quad 0<r<1,
    $$
and thus $\Phi_c$ satisfies both \eqref{111} and \eqref{112} if $c<\frac{p}{2}-1$. Further, each function $\Phi_n(r)=\log_n\frac{\exp_n(2)}{1-r}$, $n\in\N$, satisfies
    $$
    \frac{\Phi_n'(r)}{\Phi_n(r)}(1-r)\log\frac{e}{1-r}\to0,\quad
    r\to1,
    $$
and hence satisfies all hypotheses of the next result.

\begin{proposition}\label{pr:sharp}
Let $2<p<\infty$, and let $\Phi:[0,1)\to(1,\infty)$ be a
differentiable increasing unbounded function such that
$\frac{\Phi'(r)}{\Phi(r)}(1-r)$ is decreasing and \eqref{111} and
\eqref{112} are satisfied. Then there exists a positive Borel
measure $\mu$ on $\D$ such that
    \begin{equation}\label{cardpno}
    \sup_{I\subset\T}\frac{\mu\left(S(I)\right)}{|I|\left(\log\frac{e}{|I|}\right)^{-p/2+1}\Phi(1-|I|)}<\infty,
    \end{equation}
but $\mu$ is  not a $p$-Carleson measure for $\Dp$.
\end{proposition}

\begin{proof}
The radial measure
    $$
    d\mu(z)=\frac{\Phi(|z|)\,dA(z)}{(1-|z|)\left(\log\frac{e}{1-|z|}\right)^{p/2}},\quad
    z\in\D.
    $$
satisfies \eqref{cardpno} by Lemma~\ref{Lem:BernoulliHospital}. To
see that $\mu$ is not a $p$-Carleson measure for $\Dp$, consider
the lacunary series associated to $\Phi$ via \eqref{eqlac}. By Lemma \ref{integralmeans}, $f\in\Dp$ and
   \begin{equation*}
    \begin{split}
    \|f\|^p_{L^p(\mu)}&=\int_0^1\frac{M_p^p(r,f)\Phi(r)}{(1-r)\left(\log\frac{e}{1-r}\right)^{p/2}}r\,dr
    \\ & \gtrsim \int_0^1 \frac{r\Phi'(r)}{\Phi(r)}\,dr
   \gtrsim  \lim_{t\to1-}\log \Phi(t)=\infty,
    \end{split} \end{equation*}
   which finishes the proof.
\end{proof}

\end{document}